\documentclass[12pt]{article}
\usepackage{fullpage,amsthm,amsmath,amssymb,enumerate}

\newcounter{thmctr}
\newtheorem{thm}[thmctr]{Theorem}
\newtheorem{lemma}[thmctr]{Lemma}

\newtheorem{cor}[thmctr]{Corollary}
\theoremstyle{definition}
\newtheorem*{definition}{Definition}
\theoremstyle{plain}

\newcommand{\uic}{Department of Mathematics, Statisticts, and Computer Science,
University of Illinois at Chicago, 851
S. Morgan Street, Chicago, IL, USA. }

\newcommand{\dhruvfoot}{\footnote{\uic mubayi@math.uic.edu.
 Research supported in part by  NSF Grant 0969092.}}
\newcommand{\johnfoot}{\footnote{\uic lenz@math.uic.edu.}}
\newcommand{\peterfoot}{\footnote{School of Mathematical Sciences, Queen Mary, University of London,
Mile End Road, London E1 4NS, UK. p.keevash@qmul.ac.uk. Research 
supported in part by ERC grant 239696 and EPSRC grant EP/G056730/1.}}

\title{Spectral extremal problems for hypergraphs}
\author{Peter Keevash \peterfoot \and John Lenz \johnfoot \and Dhruv Mubayi \dhruvfoot}

\begin{document}
\maketitle

\begin{abstract}
In this paper we consider spectral extremal problems for hypergraphs.
We give two general criteria under which such results may be deduced
from `strong stability' forms of the corresponding (pure) extremal results.
These results hold for the $\alpha$-spectral radius defined using
the $\alpha$-norm for any $\alpha>1$; the usual spectrum is the case $\alpha=2$.

Our results imply that any hypergraph Tur\'{a}n problem which has the 
stability property and whose extremal construction satisfies some rather 
mild continuity assumptions admits a corresponding spectral result.  A 
particular example is to determine the maximum $\alpha$-spectral radius of any 
$3$-uniform hypergraph on $n$ vertices not containing the Fano plane,
when $n$ is sufficiently large. Another is to determine the maximum $\alpha$-spectral 
radius of any graph on $n$ vertices not containing some fixed colour-critical graph,
when $n$ is sufficiently large;
this generalises a theorem of Nikiforov who proved stronger results in the case $\alpha=2$. We also obtain an $\alpha$-spectral version of 
the Erd\H{o}s-Ko-Rado theorem on $t$-intersecting $k$-uniform hypergraphs.
\end{abstract}

\section{Introduction} 
\label{sec:Introduction}

Let $\mathcal{F}$ be a family of $k$-uniform hypergraphs. 
The {\em Tur\'an number} $\text{ex}(n,\mathcal{F})$ is the maximum number
of edges in a $k$-uniform hypergraph on $n$ vertices, that is {\em $\mathcal{F}$-free}, 
in that it does not have a (not necessarily induced) subgraph isomorphic to any $F \in \mathcal{F}$.  
It is a long-standing open problem in Extremal Combinatorics to develop some understanding of 
these numbers for general hypergraphs.
For ordinary graphs ($k=2$) the picture is fairly complete, but for $k \ge 3$ there are very few
known results. Tur\'an \cite{T61} posed the natural question of determining $\text{ex}(n,F)$ when
$F=K^k_t$ is a complete $k$-uniform hypergraph on $t$ vertices. 
An asymptotic solution to a (non-degenerate)
Tur\'an problem is equivalent to determining the {\em Tur\'an density} 
$\pi(\mathcal{F}) = \lim_{n \to \infty} \binom{n}{k}^{-1} \text{ex}(n,\mathcal{F})$. 
It is still an open problem to determine any value of $\pi(K^k_t)$ with $t>k>2$.
For a summary of progress on hypergraph Tur\'an problems before
2011 we refer the reader to the survey \cite{K}.

In this paper we consider spectral analogues of Tur\'an-type problems for hypergraphs. For graphs,
the picture is again fairly complete, due in large part to a longstanding project of Nikiforov. For
example, he generalised the classical theorem of Tur\'an \cite{T41}, by determining the maximum
spectral radius of any $K_{r+1}$-free graph $G$ on $n$ vertices. Here, the \emph{spectral radius}
$\lambda(G)$ of a graph $G$ is the maximum eigenvalue of its adjacency matrix. Let $T_{r,n}$ denote
the $r$-partite \emph{Tur\'an graph}, i.e.\ the complete $r$-partite graph on $n$ vertices that is
\emph{balanced}, in that its part sizes are as equal as possible. Nikiforov \cite{N2} showed that
$\lambda(G) \le \lambda(T_{r,n})$, with equality only if $G=T_{r,n}$. This generalises Tur\'an's
theorem, as the spectral radius of a graph is always at least its average degree. For many other
spectral analogues of results in extremal graph theory we refer the reader to the survey of
Nikiforov \cite{N}.

\subsection{Definitions}

We adopt the following definition of hypergraph eigenvalues introduced by Friedman and Wigderson
\cite{Fr,FW}. This is based on the extremal characterisation of the spectral radius of a graph,
namely that $\lambda(G) = \max_{\|x\|=1} 2\sum_{ij \in E(G)} x_i x_j$. First we define the
corresponding multilinear form for hypergraphs.

\begin{definition} 
  Let $H$ be a $k$-uniform hypergraph.  The
  \emph{adjacency map of $H$} is the symmetric $k$-linear map $\tau_H : W^k \rightarrow \mathbb{R}$
  defined as follows, where $W$ is the vector space over $\mathbb{R}$ of dimension $|V(H)|$.  First,
  for all $v_1, \dots, v_k \in V(H)$, let
  \begin{align*}
    \tau_H(e_{v_1}, \dots, e_{v_k}) = 
    \begin{cases}
      1 & \left\{ v_1, \ldots, v_k \right\} \in E(H), \\
      0 & \text{otherwise},
    \end{cases}
  \end{align*}
  where $e_v$ denotes the indicator vector of the vertex $v$, that is the vector which has a one in
  coordinate $v$ and zero in all other coordinates.  This defines the value of $\tau_H$ when the
  inputs are standard basis vectors of $W$, then we extend $\tau_H$ to all the domain linearly.  Alternatively, one can directly define the adjacency map by
  $$\tau_H(x_1, \ldots, x_n)=k! \sum_{\{i_1, \ldots, i_k\} \in E(H)} x_{i_1}\cdots x_{i_k}.$$
\end{definition}

Now we can define the spectral radius of a hypergraph. In fact, our results will hold for the
following more general parameter when $\alpha>1$: the definition from \cite{FW} is obtained by
setting $\alpha=2$.

\begin{definition}
  Let $H$ be a $k$-uniform hypergraph and let $\tau_H$ be the adjacency map of $H$. 
  For $\alpha \in \mathbb{R}$, the \emph{$\alpha$-spectral radius of $H$} is
  \[ \lambda_{\alpha}(H) = \max_{x: \|x\|_{\alpha} = 1} \tau_H(x,\dots,x).\]
\end{definition}

To motivate our first result, we remark that many Tur\'an problems exhibit the `stability'
phenomenon, namely that $\mathcal{F}$-free $k$-uniform hypergraphs of nearly maximal size must also be near to
an extremal example. The classical result of this type is the Erd\H{o}s-Simonovits Stability Theorem
(see \cite{Sim}), which states that any $K_{r+1}$-free graph $G$ on $n$ vertices with $e(G) =
e(T_{r,n}) + o(n^2)$ differs by $o(n^2)$ edges from $T_{r,n}$. A closely related property, known as
`strong stability', is exemplified by a result of Andr\'asfai, Erd\H{o}s and S\'os \cite{AES},
that any $K_{r+1}$-free graph $G$ on $n$ vertices with minimum degree $\delta(G) > \frac{3r-4}{3r-1}
n$ must be $r$-partite. Stability is an important phenomenon for hypergraph Tur\'an problems, as in
several cases, the only known proof uses the `stability method', which is first to prove the
stability version, and then to refine this to obtain the exact result.

In the following definition we formalise a generalised form of strong stability.
First we introduce some more notation. Let $H$ be a $k$-uniform hypergraph.
For each $0 \le s \le k-1$ we define the \emph{minimum $s$-degree} $\delta_s(H)$
as the minimum over all sets $S$ of $s$ vertices of the number of edges containing $S$.
We define the \emph{generalised Tur\'an number} $\mbox{ex}_s(n,\mathcal{F})$ as the largest value
of $\delta_s(H)$ attained by an $\mathcal{F}$-free $k$-uniform hypergraph $H$ on $n$ vertices.
Note that $\delta_0(H)=e(H)$, 
so $\mbox{ex}_0(n,\mathcal{F})=\mbox{ex}(n,\mathcal{F})$ is the usual Tur\'an number.

\begin{definition}
  Let $\mathcal{F}$ be a family of $k$-uniform hypergraphs, $n \geq 1$, $0 \le s \le k-1$ and $c > 0$.  
  We say that a family $\mathcal{G}$ of $k$-uniform, $\mathcal{F}$-free hypergraphs is
  \emph{$(\mathcal{F},n,s,c)$-universal} if for any $k$-uniform, $n$-vertex, 
  $\mathcal{F}$-free hypergraph $H$  with $\delta_s(H) > c\ \text{ex}_s(n,\mathcal{F})$ 
  there exists $G \in \mathcal{G}$ such that $H \subseteq G$.
\end{definition}

For example, if $k = 2$, $\mathcal{F} = \{ K_3\}$, $n$ is an even integer at least two, $s = 1$, $c
= \frac{4}{5}$, and $\mathcal{G}$ is the family of $n$-vertex complete bipartite graphs, then $\mathcal{G}$
is $(\mathcal{F},n,s,c)$-universal.  Indeed, $ex_1(n,K_3) = \frac{n}{2}$ and if $H$ is any
$n$-vertex, triangle-free  graph with $\delta(H) > \frac{4}{5} ex_1(n,K_3) = \frac{2n}{5}$, then by
a result of Andr\'{a}sfai, Erd\H{o}s, and S\'{o}s~\cite{AES}, $H$ must be bipartite so there exists $G$ with $H \subset G \in
\mathcal{G}$.

\subsection{Results}

Our first main result gives a general condition under which we can obtain an $\alpha$-spectral analogue of
a hypergraph Tur\'an result: we require an estimate on the difference of successive Tur\'an numbers 
and a sequence of universal families in which the
$\alpha$-spectral radius is close to what one would expect under the uniform weighting of vertices.
If $\alpha > 1$ and $\mathcal{G}$ is a family of $k$-uniform hypergraphs, we define
\[\lambda_{\alpha}(\mathcal{G}) = \sup \{ \lambda_{\alpha}(G) : G \in \mathcal{G}\}.\]

\begin{thm} \label{thm:density}
  Let $N \ge k \geq 2$, $\alpha > 1$, $\epsilon>0$
  and $\mathcal{F}$ be a family of $k$-uniform hypergraphs with $\pi(\mathcal{F})>0$.
  There exist $\delta > 0$ and $n_0 > N$ such that the following holds.

  Suppose that for all $n \ge N$ we have
  \begin{align} \label{eq:density-ex-smooth} 
     \left| \text{ex}(n,\mathcal{F}) - \text{ex}(n-1,\mathcal{F})
     - \pi(\mathcal{F}) \binom{n}{k-1} \right| < \delta n^{k-1}
  \end{align}
  and an $(\mathcal{F},n,1,1-\epsilon)$-universal family $\mathcal{G}_n$ such that
  \begin{align} \label{eq:density-lambda-good}
      \left| \lambda_{\alpha}(\mathcal{G}_n) 
        - k! \text{ex}(n,\mathcal{F}) n^{-k/\alpha} \right| 
      \leq \delta n^{k - k/\alpha - 1}.
  \end{align}
  
  Then for any $\mathcal{F}$-free $k$-uniform hypergraph $H$
  on $n \geq n_0$ vertices we have
  \begin{align*}
    \lambda_{\alpha}(H) \leq \lambda_{\alpha}(\mathcal{G}_n).
  \end{align*}
  In addition, if equality holds then $H \in \mathcal{G}_n$.
\end{thm}

As example applications of Theorem \ref{thm:density}, 
we will give $\alpha$-spectral Tur\'an results 
for colour-critical graphs and for the Fano plane. 
We say that a graph $F$ is \emph{colour-critical} if there is an edge $e$ of $F$ such that
$\chi(F-e) < \chi(F)$, where $\chi$ denotes the chromatic number.
Let $F$ be a colour-critical graph with $\chi(H)=r+1$.
Simonovits \cite{Sim} showed that there is $n_0$, such that for any $F$-free graph $G$ on $n>n_0$
vertices, $e(G) \le e(T_{r,n})$, with equality only if $G=T_{r,n}$. 
Nikiforov \cite{N5} extended this by showing that there is $n_0$, such that if $G$ has
 $n>n_0$ vertices and $\lambda_2(G) > \lambda_2(T_{r,n})$, then $G$ contains a copy of the complete $r$-partite graph with parts of size $\Omega(\log n)$ plus an extra edge (actually, in his result one of the parts can even be much larger). Our first result proves an $\alpha$-spectral version of the extremal result that guarantees a copy of $F$.

\begin{cor} \label{cor:colourcritical}
  Let $F$ be a colour-critical graph with $\chi(F)=r+1$.
  For any $\alpha > 1$, there is $n_0$, such that for any $F$-free graph $G$ on
  $n>n_0$ vertices, $\lambda_{\alpha}(G) \leq \lambda_{\alpha}(T_{r,n})$, 
  with equality if and only if $G=T_{r,n}$.
\end{cor}

The Fano plane is the $3$-uniform hypergraph where the vertices are 
the non-zero vectors in $\mathbb{F}_2^3$ and the edges are all triples $xyz$ such that 
$x+y=z$.  Let $B_n$ denote the balanced complete bipartite $3$-uniform
hypergraph on $n$ vertices: there are two parts of sizes $\lfloor n/2 \rfloor$ and $\lceil n/2
\rceil$, and the edges consist of all triples that intersect both parts. It was conjectured by S\'os
\cite{Sos}, then proved independently by Keevash and Sudakov \cite{KS} and by F\"uredi and
Simonovits \cite{FS}, that $B_n$ is the unique largest Fano-free $3$-uniform hypergraph on $n$
vertices for large $n$. The following is an $\alpha$-spectral generalisation of this result.

\begin{cor} \label{cor:fano}
  For any $\alpha > 1$, there is $n_0$, such that for any Fano-free $3$-uniform hypergraph $H$ on
  $n>n_0$ vertices, $\lambda_{\alpha}(H) \leq \lambda_{\alpha}(B_n)$, 
  with equality if and only if $H=B_n$.
\end{cor}

Our second result gives another general condition under which we can obtain 
an $\alpha$-spectral analogue of a hypergraph Tur\'an result,
namely if the universality parameter $c$ is sufficiently small.

\begin{thm} \label{thm:strongstab}
  Let $n \ge k \geq 2$, $\alpha > 1$, $c>0$ and $\mathcal{F}$ be a family of $k$-uniform hypergraphs.
  Suppose $\mathcal{G}_n$ is an $(\mathcal{F},n,0,c)$-universal family 
  of $k$-uniform hypergraphs on $n$ vertices such that
  \begin{align*}
      c < \frac{\lambda_{\alpha}(\mathcal{G}_n)^{\alpha/(\alpha-1)}}{k!\text{ex}(n,\mathcal{F})}.
  \end{align*}
  If $H$ is a $k$-uniform $\mathcal{F}$-free hypergraph on $n$ vertices then
  \begin{align*}
    \lambda_{\alpha}(H) \leq \lambda_{\alpha}(\mathcal{G}_n).
  \end{align*}
  Furthermore, if $\lambda_{\alpha}(H) > (ck!\text{ex}(n,\mathcal{F}))^{(\alpha-1)/\alpha}$
  then $H \subseteq G$ for some $G \in \mathcal{G}_n$.
\end{thm}

We deduce this from the following lemma which has independent interest.

\begin{lemma} \label{lem:edgebound}
Let $\alpha>1$ and $H$ be a $k$-uniform hypergraph with $e$ edges. Then 
\[ \lambda_{\alpha}(H) \le (k!e)^{1-1/\alpha} . \]
\end{lemma}

Another consequence of Lemma \ref{lem:edgebound} is an $\alpha$-spectral 
approximate version of the Lov\'asz form of the Kruskal-Katona theorem.
Let $H$ be a $k$-uniform hypergraph. Its \emph{shadow} $\partial H$ is the $(k-1)$-uniform
hypergraph consisting of all $(k-1)$-sets that are contained in some edge of $H$. 
Lov\'asz \cite[Ex 13.31(b)]{Lo} showed that if $e(H) = \tbinom{x}{k} = x(x-1) \cdots (x-k+1)/k!$ 
for some real number $x \ge k$ then $e(\partial H) \ge \tbinom{x}{k-1}$,
with equality if and only if $x$ is an integer and $H = K^k_x$.

\begin{cor} \label{cor:kk}
Let $\alpha>1$ and $H$ be a $k$-uniform hypergraph with 
$\lambda_{\alpha}(H) \ge \Big( k!\tbinom{x}{k} \Big)^{1-1/\alpha}$, 
for some real $x \ge k-1$. Then $e(\partial H) \ge \tbinom{x}{k-1}$.
\end{cor}

Theorem \ref{thm:strongstab} is useful for `degenerate' Tur\'an-type problems,
in which the universality parameter $c$ is often not only small,
but even tends to zero as $n$ tends to infinity.
One such case is the problem of determining the
maximum size of a $k$-uniform hypergraph $H$ on $n$ vertices that is \emph{$t$-intersecting}, in that 
$|E \cap E'| \geq t$ for all edges $E \neq E'$ of $H$. The following definition describes a natural
construction for this problem.

\begin{definition}
  A $k$-uniform hypergraph is a \emph{$t$-star} if there exists a vertex set $W$ (called the
  \emph{center}) of size $t$ such that every edge contains $W$.  The \emph{complete
  $k$-uniform, $n$-vertex, $t$-star}, denoted $S^k_{n,t}$, is the $n$-vertex $t$-star which has
  $\binom{n-t}{k-t}$ edges. 
\end{definition}

Erd\H{o}s, Ko and Rado \cite{EKR} showed that when $n$ is sufficiently large, 
$S^k_{n,t}$ is the largest $t$-intersecting $k$-uniform hypergraph on $n$ vertices.
Wilson \cite{W} showed the same result for $n \ge (t+1)(k-t+1)$,
which is best possible, as other constructions are larger for smaller $n$.
The full picture was the subject of a longstanding conjecture of Frankl, 
finally resolved by the Complete Intersection Theorem of Ahlswede and Khachatrian \cite{AK}. 
In the following we give a spectral analogue of a strong stability form of the result of Erd\H{o}s,
Ko and Rado. 

\begin{cor} \label{cor:intersecting}
  For any $k \ge 2$, $t \ge 1$ and $\alpha > 1$  
  there is $n_0$ such that the following holds for $n \ge n_0$.
  Let $H$ be an $n$-vertex, $k$-uniform, $t$-intersecting hypergraph.
  Then $\lambda_{\alpha}(H) \leq \lambda_{\alpha}(S^k_{t,n})$, 
  with equality if and only if $H = S^k_{t,n}$.
  Furthermore, there is a constant $c = c(k,t)$ such that 
  if $\lambda_{\alpha}(H) > cn^{-(\alpha-1)/\alpha} \lambda_{\alpha}(S^k_{t,n})$ then $H$ is a star.
\end{cor}

\section{Universal Families} 
\label{sec:universal}

In this section we prove our general results on universal families.
After giving some preliminary facts in the first subsection,
we prove Theorem \ref{thm:density} in the second subsection,
and Theorem \ref{thm:strongstab} in the third subsection.

\subsection{Preliminaries}

We adopt the following notation. Let $H$ be a $k$-uniform hypergraph. 
We let 
\[E_o(H) = \{ (v_1,\dots,v_k) \in V(H)^{k} : \{v_1,\dots,v_k\} \in E(H)\} \] 
be the ordered edge set of $H$.
For any vertex $u$ of $H$ we let 
\[L_o(u) = \{ (v_1,\dots,v_{k-1}) \in V(H)^{k-1} : \{v_1,\dots,v_{k-1},u\} \in E(H)\}\]
be the ordered link hypergraph of $u$.
Let $\vec{w}$ be a vector in $W = \mathbb{R}^{V(H)}$.
For any $d \geq 1$ and any tuple $S \in V(H)^d$ we write 
\[w(S) = \prod_{i \in S} w_i.\]

Next, we derive two useful inequalities that are consequences of the assumptions of Theorem
\ref{thm:density}.  Assume that $k, n, N, \alpha, \delta, \mathcal{F}$, and $\mathcal{G}_n$ satisfy
equations \eqref{eq:density-ex-smooth} and \eqref{eq:density-lambda-good} for all $n \geq N$.  We
will adopt the notation that \[\mu_n := \lambda_{\alpha}(\mathcal{G}_n).\] First, using
\eqref{eq:density-lambda-good} and the fact that the ratios $\binom{n}{k}^{-1}
\text{ex}(n,\mathcal{F})$ are decreasing with $n$ and tend to $\pi(\mathcal{F})$, we have that
\begin{align} \label{eq:mu}
 \mu_n = (1+o(1)) \pi(\mathcal{F}) n^{k - k/\alpha}.
\end{align}
Next we want to bound the gap between $\mu_n$ and $\mu_{n-1}$.  Let $M = (k-\frac{k}{\alpha})
\pi(F) n^{k - k/\alpha-1}$; we will show that $\mu_n - \mu_{n-1}$ is close to $M$.  First, use
the triangle inequality and \eqref{eq:density-lambda-good} to obtain
\begin{align}
  |\mu_n &- \mu_{n-1} - M | \nonumber \\
  &\leq \left| k! ex(n,\mathcal{F})n^{-k/\alpha} - k! ex(n-1,\mathcal{F}) (n-1)^{-k/\alpha}
  -M \right| +
  2\delta n^{k - k/\alpha - 1} \nonumber \\
  &= k! n^{-k/\alpha} \left| ex(n,\mathcal{F}) - ex(n-1,\mathcal{F}) \left( \frac{n-1}{n}
  \right)^{-k/\alpha} - \frac{M n^{k/\alpha}}{k!} \right| + 2\delta n^{k-k/\alpha-1}.
  \label{eq:gapofexp}
\end{align}
Using \eqref{eq:density-ex-smooth} and the fact that for large $n$, $(\frac{n-1}{n})^{-k/\alpha} = 1
+ \frac{k}{\alpha n} \pm O(\frac{1}{n^2})$,
\begin{align*}
  &\left| ex(n,\mathcal{F}) - ex(n-1,\mathcal{F}) \left( \frac{n-1}{n}
  \right)^{-k/\alpha} - \frac{Mn^{k/\alpha}}{k!} \right| \\
  & \quad \quad \leq
  \left| ex(n,\mathcal{F}) - \left( ex(n,\mathcal{F}) - \pi(F) \binom{n}{k-1} \right) \left( 1 +
  \frac{k}{\alpha n} \right) - \frac{Mn^{k/\alpha}}{k!} \right| + 2\delta n^{k-1} \\
  & \quad \quad \leq \left| \frac{-k}{\alpha n} ex(n,\mathcal{F}) + \pi(F) \binom{n}{k-1} - \frac{M
  n^{k/\alpha}}{k!} \right| + 3 \delta n^{k-1} \\
  \intertext{where the last inequality used that $\frac{1}{n} \binom{n}{k-1} = o(n^{k-1})$.  Now using that
$\binom{n}{k}^{-1} ex(n,\mathcal{F})$ converges to $\pi(F)$, for large $n$ the above inequality
continues as}
  &\quad \quad \leq \left| \frac{-k}{\alpha n} \pi(F) \frac{n^{k}}{k!} + k \pi(F) \frac{n^{k-1}}{k!} - \frac{M
  n^{k/\alpha}}{k!}  \right| + 4 \delta n^{k-1}. \nonumber
\end{align*}
By the definition of $M$, the expression inside the above absolute value is zero.  Therefore,
\eqref{eq:gapofexp} simplifies to
\begin{align*}
  |\mu_n - \mu_{n-1} - M| \leq k! n^{-k/\alpha} \left( 4\delta n^{k-1} \right) + 2\delta
  n^{k-k/\alpha-1} \leq 5k!\delta n^{k-k/\alpha - 1}.
\end{align*}
In summary, we have proved that
\begin{align} \label{eq:mudiff}
  \left| \mu_n - \mu_{n-1} - \left( k - \frac{k}{\alpha} \right) \pi(\mathcal{F}) n^{k - k/\alpha -
  1} \right| \leq 5k!\delta n^{k-k/\alpha-1}.
\end{align}

Next, we will estimate $ex_1(n,\mathcal{F})$.  Indeed, it is easy to see that
$\frac{n}{k} ex_1(n,\mathcal{F}) \leq ex(n,\mathcal{F})$ so that
\begin{align*}
  ex_1(n,\mathcal{F}) \leq (1+o(1)) \frac{k}{n} \pi(F) \binom{n}{k} = (1+o(1)) \pi(F)
  \binom{n}{k-1}.
\end{align*}
On the other hand, given a hypergraph $H$ with $(\pi(F) + \epsilon) \binom{n}{k}$ edges, we can
delete its vertices of small degree obtaining a subhypergraph $H'$ on $m \geq \epsilon^{1/k}n$
vertices with minimum degree at least $\pi(F) \binom{m}{k-1}$ (see e.g.\ \cite[p. 121]{Bollobas}
for details in the case $k=2$).  Therefore,
\begin{align} \label{eq:exoneispidense}
  ex_1(n,\mathcal{F}) = (1+o(1)) \pi(F) \binom{n}{k-1}.
\end{align}

Finally, in our calculations we will frequently use H\"older's inequality that
$|x \cdot y| \le \|x\|_p \|y\|_q$ when $p,q>1$ and $p^{-1}+q^{-1}=1$,
and Bernoulli's inequality, which is as follows.
Suppose that $t>-1$ and $a \in \mathbb{R}$. 
If $a \le 0$ or $a \ge 1$ then $(1+t)^a \ge 1+at$;
if $0 \le a \le 1$ then $(1+t)^a \le 1+at$.

\subsection{Proof of Theorem \ref{thm:density}}

Throughout this subsection we let $H$ be a $k$-uniform, $\mathcal{F}$-free hypergraph 
on $n$ vertices, and $\vec{w}$ be a vector such that
\[ \tau_H(\vec{w},\dots,\vec{w}) = \lambda_{\alpha}(H) 
\text{ and } \left\lVert \vec{w} \right\rVert_{\alpha} = 1. \] 

\begin{lemma} \label{lem:neighborsum}
For all $1 \leq i \leq n$,
\begin{align*} 
  \sum_{S \in L_o(i)} w(S) = \tau_H(e_i, \vec{w},\dots,\vec{w}) = \lambda_{\alpha}(H) w_i^{\alpha-1}.
\end{align*}
\end{lemma}

\begin{proof} 
First, expanding the definition of $\tau_H$, we have that
\begin{align*}
  \tau_H(e_i, \vec{w}, \dots, \vec{w}) &= \tau_H\left(e_i, \sum_{j_2} w_{j_2} e_{j_2}, \dots,
  \sum_{j_k} w_{j_k} e_{j_k} \right)  \\
  &= \sum_{j_2,\dots,j_k} w_{j_2} \cdots w_{j_k}
  \tau_H(e_i, e_{j_2}, \dots, e_{j_k}) = \sum_{S \in L_o(i)} w(S).
\end{align*}

Now define the following two functions:
\begin{align*}
  f(x) &= \tau_H(x,\dots,x) = \sum_{(v_1,\dots,v_k) \in E_o(H)} x_{v_1}\cdots x_{v_k}\\
  g(x) &= \sum_i x_i^\alpha.
\end{align*}
By definition, $\vec{w}$ maximizes $f(x)$ subject to $g(x) = 1$, so the method of Lagrange multipliers implies
there exists a constant $\Lambda$ such that for all $1 \leq i \leq n$,
\begin{align*}
  \left. \frac{\partial f}{\partial x_i} \right|_{\vec{w}} &= 
  \Lambda \left. \frac{\partial g}{\partial x_i} \right|_{\vec{w}}.
\end{align*}
Note that
\begin{align*}
  \frac{\partial f}{\partial x_i} &= \frac{\partial}{\partial x_i} \sum_{(v_1,\dots,v_k) \in E_o(H)}
    x_{v_1} \cdots x_{v_k} \\
    &= k \sum_{(v_1,\dots,v_{k-1}) \in L_o(i)} x_{v_1}\cdots x_{v_{k-1}} \\
    &= k \, \tau_H(e_i, x, \dots, x).
\end{align*}
Therefore
\begin{align}
  \left. \frac{\partial f}{\partial x_i} \right|_{\vec{w}} &=
  k \, \tau_H(e_i,\vec{w},\dots,\vec{w}) = \Lambda \alpha w_i^{\alpha-1}. \label{eq:lagrange}
\end{align}
Now
\begin{align*}
  \lambda_{\alpha}(H) = \tau_H(\vec{w},\dots,\vec{w}) = \sum_{1 \leq i \leq n} w_i \tau_H(e_i,
  \vec{w},\dots,\vec{w}) = \sum_{1 \leq i \leq n} \frac{\Lambda \alpha}{k} w_i^{\alpha} =
  \frac{\Lambda \alpha}{k}.
\end{align*}
The lemma follows by inserting $\Lambda = \lambda_{\alpha}(H) k/\alpha$ into \eqref{eq:lagrange}.
\end{proof} 

Recall that $\mu_n = \lambda_{\alpha}(\mathcal{G}_n)$.

\begin{lemma} \label{lem:smalldegree}
  Suppose $\epsilon > 0$, $0 < \epsilon' < \epsilon (\alpha-1)/\alpha^2$, $n$
  is sufficiently large, and $\lambda_{\alpha}(H) \ge \mu_n$. If
  \begin{align*}
    \delta(H) \le (1 - \epsilon) \pi(\mathcal{F}) \binom{n}{k-1},
  \end{align*}
  then there exists a coordinate $1 \leq u \leq n$ such that
  \begin{align*}
    w_u < \frac{1 - \epsilon'}{n^{1/\alpha}}.
  \end{align*}
\end{lemma}

\begin{proof} 
Assume towards a contradiction that for every $u$, $w_u \geq (1-\epsilon')n^{-1/\alpha}$.
Let $u$ be a vertex of minimum degree. By Lemma \ref{lem:neighborsum} and Holder's Inequality,
\begin{align} \label{eq:mainboundonlambda1}
  \lambda_{\alpha}(H) w_u^{\alpha-1} =   \sum_{S \in L_o(u)} w(S) 
  \leq |L_o(u)|^{(\alpha-1)/\alpha} \left( \sum_{S \in L_o(u)} w(S)^{\alpha} \right)^{1/\alpha}.
\end{align}
Now
\begin{align*}
  \sum_{S \in L_o(u)} w(S)^{\alpha} 
  &= \sum_{S \in V(H)^{k-1}} w(S)^{\alpha} - \sum_{S \notin L_o(u)} w(S)^{\alpha} \\
  &= \left( \sum_i w_i^\alpha \right)^{k-1} - \sum_{S \notin L_o(u)} w(S)^{\alpha} \\
  &= 1 - \sum_{S \notin L_o(u)} w(S)^{\alpha}.
\end{align*}
Since $w_i \geq (1-\epsilon') n^{-1/\alpha}$ for all $i$, we obtain
\begin{align*}
  \sum_{S \in L_o(u)} w(S)^{\alpha} 
  &\leq 1 - |\overline{L_o(u)}| \left( \frac{1-\epsilon'}{n^{1/\alpha}} \right)^{(k-1)\alpha}.
\end{align*}
Since $u$ was chosen with minimum degree, $\epsilon > 0$ is fixed, $H$ is $\mathcal{F}$-free, and
$n$ is sufficiently large,
\begin{align} \label{eq:sizeoforderedlink}
   |L_o(u)| = (k-1)! \delta(H) \le (1 - \epsilon) \pi(\mathcal{F}) n^{k-1}.
\end{align}
Thus
\begin{align} \label{eq:boundonsumoverorderedlink}
  \sum_{S \in L_o(u)} w(S)^{\alpha} 
  &\leq 1 - (1 - \pi(\mathcal{F}) + \epsilon \pi(\mathcal{F})) n^{k-1} 
  \left( \frac{1-\epsilon'}{n^{1/\alpha}} \right)^{(k-1)\alpha} \nonumber \\
  &\leq 1 - (1 - \pi(\mathcal{F}) + \epsilon \pi(\mathcal{F})) (1-\epsilon') \le \pi(\mathcal{F}). 
\end{align}
(The last inequality used $\epsilon' < \epsilon/2$ which follows from $\alpha \geq 1$.)
By \eqref{eq:mu}, for large enough $n$ we have
\[\lambda_{\alpha}(H) \ge \mu_n \ge (1-\epsilon')\pi(\mathcal{F}) n^{k(\alpha-1)/\alpha}.\]
Combining this with \eqref{eq:mainboundonlambda1} and \eqref{eq:boundonsumoverorderedlink}, we have
that
\begin{align*}
  (1-\epsilon') \pi(\mathcal{F}) n^{k(\alpha-1)/\alpha} w_u^{\alpha-1} \leq 
  \lambda_{\alpha}(H) w_u^{\alpha - 1} \leq |L_o(u)|^{(\alpha-1)/\alpha} \pi(\mathcal{F})^{1/\alpha}.
\end{align*}
Using \eqref{eq:sizeoforderedlink}, we obtain
\begin{align*}
  (1-\epsilon') \pi(\mathcal{F}) n^{k(\alpha-1)/\alpha} w_u^{\alpha - 1} \leq
    \Big( (1 - \epsilon) \pi(\mathcal{F}) n^{k-1} \Big)^{(\alpha - 1)/\alpha} \pi(\mathcal{F})^{1/\alpha}.
\end{align*}
Using $w_u \geq (1-\epsilon') n^{-1/\alpha}$ and cancelling 
$\pi(\mathcal{F}) n^{(k-1)(\alpha - 1)/\alpha}$ from both sides, we obtain
\begin{align*}
  (1 - \epsilon')^{\alpha} \leq (1 - \epsilon)^{(\alpha - 1)/\alpha}.
\end{align*}
Then by Bernoulli's inequality,
\[ 1-\epsilon' \le (1 - \epsilon)^{(\alpha - 1)/\alpha^2} \le 1 - \epsilon (\alpha - 1)/\alpha^2, \]
which is a contradiction.
\end{proof} 

\begin{lemma} \label{lem:deletevertex}
  Suppose $n$ is sufficiently large and $u$ is such that $w_u < (1-\epsilon')n^{-1/\alpha}$. Then
  \[ \lambda_{\alpha}(H - u) \ge 
  (1 - (1-1/\alpha)k (1-\epsilon'/2) n^{-1}) \lambda_{\alpha}(H).\] 
  Suppose also that $\lambda_{\alpha}(H) = \mu_n + C$, where $C \ge 0$, and that 
  $\delta \le \epsilon'(1-1/\alpha)\pi(\mathcal{F})/20k!$ 
  and $C \le k^{-1}\delta n^{k(\alpha-1)/\alpha}$. Then
  \begin{align*}
    \lambda_{\alpha}(H - u) \geq \mu_{n-1} + C + \delta n^{k(\alpha-1)/\alpha - 1}.
  \end{align*}
\end{lemma}

\begin{proof} 
Let $H'$, $\vec{w}'$ be the restrictions of $H$, $\vec{w}$ to $V(H) - u$. We can write 
\begin{align*}
 \lambda_{\alpha}(H) & = \tau_H(\vec{w},\dots,\vec{w}) \\
&  = \tau_{H'}(\vec{w}',\dots,\vec{w}') + k w_u \sum_{S \in L_o(u)} w(S) \\
& = \tau_{H'}(\vec{w}',\dots,\vec{w}') + k \lambda_{\alpha}(H) w_u^{\alpha} 
\end{align*}
by Lemma \ref{lem:neighborsum}. Next, note that 
\[\left\lVert \vec{w}' \right\rVert_{\alpha} = (1-w_u^{\alpha})^{1/\alpha}.\]
Letting $\vec{w}^* = (1-w_u^{\alpha})^{-1/\alpha} \vec{w}'$ be the unit vector in direction
$\vec{w}'$, we have
\begin{align*}
\lambda_{\alpha}(H - u) & \geq \tau_{H'}(\vec{w}^*,\dots,\vec{w}^*)
= (1-w_u^{\alpha})^{-k/\alpha} \tau_{H'}(\vec{w}',\dots,\vec{w}') \\
& = (1-w_u^{\alpha})^{-k/\alpha} (1-kw_u^{\alpha}) \lambda_{\alpha}(H).
\end{align*}
By Bernoulli's inequality and $w_u < (1-\epsilon')n^{-1/\alpha}$ we have
\begin{align*}
 (1-w_u^{\alpha})^{-k/\alpha} (1-kw_u^{\alpha}) 
 & \ge  (1 + (k/\alpha)w_u^{\alpha}) (1-kw_u^{\alpha}) \\
 & \ge  1 - (1-1/\alpha)k (1-\epsilon')^{\alpha} n^{-1} - k^2 n^{-2} \\
 & \ge  1 - (1-1/\alpha)k (1-\epsilon'/2) n^{-1},
\end{align*}
which proves the first statement of the lemma.
Next we substitute $\lambda_{\alpha}(H) = \mu_n + C$
and use \eqref{eq:mudiff} to obtain
\[ \lambda_{\alpha}(H - u) \geq 
(1 - (1-1/\alpha)k (1-\epsilon'/2) n^{-1})
\Big(\mu_{n-1} + C + (k(1-1/\alpha)\pi(\mathcal{F})-5k!\delta) n^{k(\alpha-1)/\alpha-1} \Big). \]
By \eqref{eq:mu}, for large enough $n$ we have
\[\mu_n \le (1+\delta/k)\pi(\mathcal{F}) n^{k(\alpha-1)/\alpha}.\]
Using $\delta \le \epsilon'(1-1/\alpha)\pi(\mathcal{F})/20k!$
and $C \le k^{-1}\delta n^{k(\alpha-1)/\alpha}$ we deduce
\begin{align*}
 \lambda_{\alpha}(H - u) & \geq 
 \mu_{n-1} - (1-1/\alpha)k (1-\epsilon'/2) n^{-1} (1+\delta/k) \pi(\mathcal{F}) n^{k(\alpha-1)/\alpha} \\
& \quad   + (1-k/n) \Big(C + (k(1-1/\alpha)\pi(\mathcal{F})-5k!\delta) n^{k(\alpha-1)/\alpha-1} \Big)\\
& \ge \mu_{n-1} + C 
 + \Big(k(1-1/\alpha) \epsilon' \pi(\mathcal{F})/2 - 7k!\delta \Big) n^{k(\alpha-1)/\alpha-1} \\
& \ge \mu_{n-1} + C + \delta n^{k(\alpha-1)/\alpha-1},
\end{align*}
as required.
\end{proof} 

\begin{proof}[Proof of Theorem~\ref{thm:density}] 
Fix $\epsilon>0$, $\epsilon' = \epsilon (\alpha-1)/8\alpha^2$, 
and $\delta = \epsilon'(1-1/\alpha)\pi(\mathcal{F})/20k!$.
We can assume that $N$ is sufficiently large to apply 
Lemmas \ref{lem:smalldegree} and  \ref{lem:deletevertex}, and in addition $N$ is sufficiently large
so that for $m > N$, by \eqref{eq:exoneispidense}, $\text{ex}_1(m,\mathcal{F}) = (1\pm
\frac{\epsilon}{2})\pi(\mathcal{F}) \binom{m}{k-1}$.
Let $n_1 = (N^k k e^{k^2}/\delta)^{2/k\epsilon'(1-1/\alpha)}$ and $n_0 = n_1e^{n_1^k/\delta}$.
Let $\{\mathcal{G}_n\}_{n\rightarrow\infty}$ be as in the statement of the theorem.
Suppose that $n \geq n_0$ and $H$ is a $k$-uniform, $n$-vertex, $\mathcal{F}$-free hypergraph
with $\lambda_{\alpha}(H) \geq \mu_n := \lambda_{\alpha}(\mathcal{G}_n)$.
We construct a sequence $H = H_n, H_{n-1}, \dots$, where $H_i$ is a $k$-uniform, $i$-vertex,
$\mathcal{F}$-free hypergraph with $\lambda_{\alpha}(H_i) = \mu_i + C_i$,
where $C_n \ge 0$ and $C_i \ge C_{i+1} + \delta (i+1)^{k(\alpha-1)/\alpha - 1}$ for $i<n$. 
To do so, suppose we have reached such an $H_i$ for some $i \ge N$, 
and that $\delta(H_i) \le (\pi(\mathcal{F}) - \epsilon/2)\binom{i}{k-1}$. 
By definition, there is a vector $\vec{w}$ such that 
\[ \tau_{H_i}(\vec{w},\dots,\vec{w}) = \lambda_{\alpha}(H_i) 
\text{ and } \left\lVert \vec{w} \right\rVert_{\alpha} = 1. \] 
Let us apply Lemma~\ref{lem:smalldegree} with input $\epsilon/2$ and $\epsilon'$ and note that by
definition of $\epsilon'$, we have that
$\epsilon' < \min\{ \epsilon (\alpha-1)/4\alpha^2, \epsilon/4\}$.  Since $\delta(H_i) \leq (\pi(F) -
\epsilon/2) \binom{i}{k-1}$, Lemma~\ref{lem:smalldegree} implies that there
is a coordinate $1 \leq u \leq i$ such that
$w_u < \frac{1 - \epsilon'}{i^{1/\alpha}}$. We set $H_{i-1} = H_i - u$.
By Lemma \ref{lem:deletevertex} we have
\[ \lambda_{\alpha}(H_{i-1}) \ge 
  (1 - (1-1/\alpha)k (1-\epsilon'/2) i^{-1}) \lambda_{\alpha}(H_i),\]
and if $C_i \le k^{-1}\delta i^{k(\alpha-1)/\alpha}$ then 
\begin{equation} \label{increase} \lambda_{\alpha}(H_{i-1}) \geq \mu_{i-1} + C_i + \delta i^{k(\alpha-1)/\alpha - 1}.\end{equation}
We claim that this process terminates at some $H_t$ with $t>N$.
Suppose towards a contradiction that the sequence reaches $H_{N}$.
First, there must be some $m>n_1$ such that $C_m \ge k^{-1}\delta m^{k(\alpha-1)/\alpha}$, 
otherwise we would have the contradiction
\[ \lambda_{\alpha}(H_{n_1}) \ge \mu_n + \sum_{i=n_1}^n \delta i^{k(\alpha-1)/\alpha - 1} 
> \delta \sum_{i=n_1}^n 1/i > \delta \log (n_0/n_1) = n_1^k.\]
Now for this $m > n_1$ with $C_m \geq k^{-1}\delta m^{k(\alpha-1)/\alpha}$, we have that
\begin{align*}
  \lambda_{\alpha}(H_{N}) 
  & \ge \lambda_{\alpha}(H_m) \prod_{i=N+1}^m (1 - (1-1/\alpha)k (1-\epsilon'/2) i^{-1}) \\
  & \ge C_m \exp \left( - \sum_{i=N+1}^m  \Big( (1-1/\alpha)k (1-\epsilon'/2) i^{-1} + (k/i)^2 \Big)
  \right) \\ 
  & \ge k^{-1}\delta m^{k(\alpha-1)/\alpha} 
  \exp \Big(-  (1-1/\alpha)k (1-\epsilon'/2) \log (m/N) - k^2 \Big) \\
  &= k^{-1}\delta m^{k(\alpha-1)/\alpha} \left(\frac{m}{N}\right)^{-(1-1/\alpha)k (1-\epsilon'/2)} e^{-k^2} \\
  & > k^{-1}\delta m^{k(\alpha-1)/\alpha} m^{-(1-1/\alpha)k (1-\epsilon'/2)} e^{-k^2}. \\
  \intertext{Now use that $m > n_1$ so that the above inequality continues as}
  & \ge k^{-1}\delta n_1^{(1-1/\alpha)k \epsilon'/2} e^{-k^2} = N^k.
\end{align*}

It is impossible for the eigenvalue of any $N$-vertex $k$-uniform hypergraph to be at least $N^k$, so this
contradiction shows that the process terminates at some $H_t$ with $t>N$.  By construction,
$\delta(H_t) > (\pi(\mathcal{F}) - \epsilon/2)\binom{t}{k-1} > (1-\epsilon)
\text{ex}_1(t,\mathcal{F})$.  Since $\mathcal{G}_t$ is $(\mathcal{F},t,1,1-\epsilon)$-universal,
there is $G_t \in \mathcal{G}_t$ such that $H_t \subseteq G_t$.  However, $\lambda_{\alpha}(H_t)
\ge \mu_t \ge \lambda_{\alpha}(G_t)$, so equality must hold. If $t<n$, then (\ref{increase}) contradicts this and therefore  $t=n$
and $H = H_n = G_n \in \mathcal{G}_n$.
\end{proof} 

\subsection{Proof of  Theorem \ref{thm:strongstab}}

In this subsection we prove Theorem \ref{thm:strongstab}.
We start with the proof of Lemma \ref{lem:edgebound}.

\begin{proof}[Proof of Lemma \ref{lem:edgebound}] 
By the power mean inequality we have
$$\tau(x, \ldots, x)=k! \sum_{\{i_1, \ldots, i_k\} \in E(G)}x_{i_1}\cdots x_{i_k} \le k! e \left( \frac{1}{e} \sum_{\{i_1, \ldots, i_k\} \in E(G)}x_{i_1}^{\alpha}\cdots x_{i_k}^{\alpha}\right)^{1/\alpha}.$$
Also, by Maclaurin's inequality we have
\[ \binom{n}{k}^{-1} \sum_{\{i_1, \ldots, i_k\} \in E(G)}x_{i_1}^{\alpha}\cdots x_{i_k}^{\alpha}
\le \binom{n}{k}^{-1} \sum_{\{i_1, \ldots, i_k\} \subseteq V(G)}x_{i_1}^{\alpha}\cdots x_{i_k}^{\alpha}
\le \left( n^{-1} \sum_{i \in V(G)} x_i^\alpha \right)^k = n^{-k}.\]
Consequently,
$$\tau(x, \ldots, x) \le k!e^{1-1/\alpha} \left(\frac{ {n \choose k} }{ n^k} \right) ^{1/\alpha} \le (k!e)^{1-1/\alpha}$$
and the proof is complete. \end{proof} 

\begin{proof}[Proof of Theorem \ref{thm:strongstab}] 
Suppose that $H$ is a $k$-uniform, $\mathcal{F}$-free hypergraph on $n$ vertices
with $\lambda_{\alpha}(H) > L := (ck!\text{ex}(n,\mathcal{F}))^{(\alpha-1)/\alpha}$.
By Lemma \ref{lem:edgebound}, the number of ordered edges in $H$
is at least $\lambda_{\alpha}(H)^{\alpha/(\alpha-1)} > ck!\text{ex}(n,\mathcal{F})$.
Since $\mathcal{G}_n$ is $(\mathcal{F},n,0,c)$-universal,
$H \subseteq G$ for some $G \in \mathcal{G}_n$.
This proves the second statement of the theorem.
Since $L < \lambda_{\alpha}(\mathcal{G}_n)$ by the bound on $c$,
we also have the first statement.
\end{proof} 

\section{Applications} 
\label{sec:apps}

In this section we apply the general results of the previous section 
to obtain spectral versions of various Tur\'an type problems for graphs and hypergraphs.
We start by showing how to exploit symmetries of a hypergraph when computing 
its $\alpha$-spectral radius (see \cite{HK} for the same argument in the case $\alpha=1$).

\begin{lemma} \label{lem:transposition}
Let $\alpha \geq 1$ and $\mathcal{F}$ be a $k$-uniform hypergraph on $[n]$. 
Suppose that the transposition $(ij)$ is an automorphism of $\mathcal{F}$. 
Consider any $\vec{w} \in \mathbb{R}^n$ with $\|\vec{w}\|_{\alpha} = 1$ 
and $w_t \ge 0$ for all $1 \le t \le n$. 
Define $\vec{w}'$ by $w'_i=w'_j=((w_i^\alpha+w_j^\alpha)/2)^{1/\alpha}$
and $w'_t=w_t$ for $t \in [n] \setminus \{i,j\}$.
Then $\tau_H(\vec{w}',\dots,\vec{w}') \ge \tau_H(\vec{w},\dots,\vec{w})$.
\end{lemma}

\begin{proof} 
Expanding the definition of
$\tau_H$, we have that
\begin{align*}
  \tau_{\mathcal{F}}(\vec{w}',\dots,\vec{w}') - \tau_{\mathcal{F}}(\vec{w},\dots,\vec{w})
  & = k!\sum_{\stackrel{F \in E(\mathcal{F})}{i \in F, j \notin F}} (w'_i - w_i) 
      \prod_{t \in F \setminus \{i\}} w_t \\
  &+   k!\sum_{\stackrel{F \in E(\mathcal{F})}{i \notin F, j \in F}} (w'_j - w_j) 
      \prod_{t \in F \setminus \{j\}} w_t \\
  &+ k!\sum_{\stackrel{F \in E(\mathcal{F})}{i \in F, j \in F}} 
       \left( w'_iw'_j - w_i w_j\right) \prod_{t \in F \setminus \{i,j\}} w_t.
\end{align*}
Since $(ij)$ is an automorphism of $\mathcal{F}$, the link graphs of $i$ and $j$ are identical so
that
\begin{align*}
  \tau_{\mathcal{F}}(\vec{w}',\dots,\vec{w}') - \tau_{\mathcal{F}}(\vec{w},\dots,\vec{w}) 
  &= k!(w'_i + w'_j - w_i - w_j) \sum_{\stackrel{A \in \binom{V(\mathcal{F})}{k-1}}{A \cup \{i\} \in
  E(\mathcal{F})}} \prod_{t \in A} w_t \\
  &+ k!(w'_iw'_j - w_iw_j) 
  \sum_{\stackrel{F \in E(\mathcal{F})}{i \in F, j \in F}} \prod_{t \in F \setminus \{i,j\}} w_t.
\end{align*}
Since $\alpha \geq 1$, by the inequality of power means we have $w'_i=w'_j \ge \sqrt{w_iw_j}$ and $w'_i + w'_j = 
2w'_i \geq w_i + w_j$, so that both of the terms in the above expression are non-negative,
completing the proof.
\end{proof} 

\begin{cor} \label{cor:sym}
Let $\alpha \geq 1$ and $\mathcal{F}$ be a $k$-uniform hypergraph on $[n]$. 
Let $\mathcal{P}_{\mathcal{F}}$ be the partition of $[n]$ into equivalence classes 
of the relation in which $i \sim j$ iff $(ij)$ is an automorphism of $\mathcal{F}$. 
Then there is $\vec{w} \in \mathbb{R}^n$ with $\|\vec{w}\|_{\alpha} = 1$,
and $w_t \ge 0$ for all $1 \le t \le n$ such that
$\tau_H(\vec{w},\dots,\vec{w}) = \lambda_{\alpha}(H)$
and $w$ is constant on each part of $\mathcal{P}_{\mathcal{F}}$.
\end{cor}

\begin{proof} 
Given $\vec{w} \in \mathbb{R}^n$ and $P \in \mathcal{P}_{\mathcal{F}}$, let 
\[\bar{w}_P = \left( |P|^{-1} \sum_{i \in P} w_i^\alpha \right)^{1/\alpha}.\] 
Consider $\vec{w}$ that minimises 
\[S(\vec{w}) = \sum_{P \in \mathcal{P}_{\mathcal{F}}} \sum_{i \in P} |w_i - \bar{w}_P|.\] 
We claim $S(\vec{w})=0$, i.e.\ $\vec{w}$ is constant on each part of $\mathcal{P}_{\mathcal{F}}$. 
For suppose not, and consider some $P \in \mathcal{P}_{\mathcal{F}}$ and $i,j$ in $P$ with $w_i \ne w_j$. 
Define $\vec{w}'$ as in Lemma \ref{lem:transposition}; 
then $\|\vec{w}'\|_{\alpha} = 1$, $\tau_H(\vec{w}',\dots,\vec{w}') = \lambda_{\alpha}(H)$
and $S(\vec{w}')<S(\vec{w})$, contradicting the choice of $\vec{w}$.
\end{proof} 

Next we estimate the $\alpha$-spectral radii for the examples that we will consider,
namely the star, the balanced bipartite $3$-graph, and the Tur\'an graph.

\begin{lemma} \label{lem:radii} Let $\alpha>1$.
\begin{enumerate}[(i)]
\item $\lambda_{\alpha}(S^k_{t,n}) = k! \binom{n-t}{k-t} 
k^{-k/\alpha} \left( \tfrac{k-t}{n-t} \right)^{(k-t)/\alpha}$.
\item $\lambda_{\alpha}(B_n) = (1+O(n^{-2})) 6e(B_n) n^{-3/\alpha}$.
\item $\lambda_{\alpha}(T_{r,n}) = (1+O(n^{-2})) 2e(T_{r,n}) n^{-2/\alpha}$.
\end{enumerate}
\end{lemma}

\begin{proof} 
For (i), note that for every pair of vertices in the center there is an automorphism
interchanging them, and the same is true for every pair of vertices not in the center.
By Corollary \ref{cor:sym} we can choose $\vec{w}$ with
$\lambda_{\alpha}(S^k_{t,n}) = \tau_H(\vec{w},\dots,\vec{w})$
such that for some $0 \le a \le 1$
we have $w_i = \left( \tfrac{a}{t} \right)^{1/\alpha}$ for $i$ in the centre
and $w_i = \left( \tfrac{1-a}{n-t} \right)^{1/\alpha}$ for $i$ not in the centre.
Then 
\[ \tau_H(\vec{w},\dots,\vec{w}) = k!\binom{n-t}{k-t}
\left( \frac{a}{t} \right)^{t/\alpha} \left( \frac{1-a}{n-t} \right)^{(k-t)/\alpha}.\]
By differentiating with respect to $a$, 
we see that $a^{t/\alpha} (1-a)^{(k-t)/\alpha}$ is maximised when
$t(1-a)/\alpha = (k-t)a/\alpha$, i.e.\ $a=t/k$.
This gives the stated formula.

Next we note a straightforward argument for the cases that $2 \mid n$ in (ii) or $r \mid n$ in (iii). 
In these cases, for every pair of vertices there is an automorphism interchanging them, 
so by Corollary \ref{cor:sym}, defining $\vec{w}$ by $w_i = n^{-1/\alpha}$ for all $i$ we have
\[\lambda_{\alpha}(H) = \tau_H(\vec{w},\dots,\vec{w}) = k! e(H) n^{-k/\alpha},\]
where $H=B_n$ and $k=3$ in (ii), or $H=T_{r,n}$ and $k=2$ in (iii).

Next consider (ii) when $n=2t+1$ is odd.
Let $V_1$ be the part of size $t$ and $V_2$ the part of size $t+1$.
By Corollary \ref{cor:sym} we can choose $\vec{w}$ with
$\lambda_{\alpha}(B_{2t+1}) = \tau_H(\vec{w},\vec{w},\vec{w})$
such that $w_i=x^{1/\alpha}$ for $i \in V_1$ and $w_i=y^{1/\alpha}$ for $i \in V_2$,
where $x,y \ge 0$ and $\|w\|_\alpha^\alpha=tx+(t+1)y=1$. Then
\begin{align*}
 \tau_{B_{2t+1}}(\vec{w},\vec{w},\vec{w}) 
  &= 6 (t+1)\binom{t}{2} x^{2/\alpha} y^{1/\alpha}
  + 6 t\binom{t+1}{2} x^{1/\alpha} y^{2/\alpha} \\
  &= 3 t(t+1) (xy)^{1/\alpha}  \Big( (t-1)x^{1/\alpha} + ty^{1/\alpha} \Big).
\end{align*}
Write \[x = \frac{1+a}{2t+1}, \ y = \frac{1-at/(t+1)}{2t+1}
\text{ and } f(a) = (xy)^{1/\alpha} ((t-1)x^{1/\alpha} + ty^{1/\alpha})\]
for some $a \in (-1,1+1/t)$. Note that for all $a$, $\left\lVert \vec{w} \right\rVert_{\alpha}^{\alpha} 
= tx + (t+1)y = 1$ so that the maximum of $\tau_{B_{2t+1}}(\vec{w},\cdots,\vec{w})$, i.e.\
$\lambda_{\alpha}(B_{2t+1})$, will be achieved by the $a$ which maximizes $f(a)$. First, when $a =
0$ then $x = y$ so $f(0) = (2t-1)(2t+1)^{-3/\alpha}$.  Next we calculate $\frac{dx}{da} =
(2t+1)^{-1}$ and $\frac{dy}{da} = \frac{-t}{(t+1)(2t+1)}$ and
\[ f'(a) = \frac{(xy)^{1/\alpha}}{(2t+1)\alpha} 
   \Big( A + B ((t-1)x^{1/\alpha} + ty^{1/\alpha}) \Big), \]
where
\[ A = (t-1)x^{-1+1/\alpha} - \tfrac{t^2}{t+1}y^{-1+1/\alpha}
   \text{ and } B = x^{-1}-\tfrac{t}{t+1}y^{-1}. \]
We will show that there are $a_0=-1/2t^2 + O(1/t^3)$ and $a_1=1/2t$
such that $f'(a) \ge 0$ for $a \le a_0$ and $f'(a) \le 0$ for $a \ge a_1$.
First we note that $B \ge 0 \Leftrightarrow 1+1/t \ge g(a)$, where
\[ g(a) = x/y = \frac{1+a}{1-at/(t+1)}.\]
Since $g(a)$ is an increasing function of $a$ on $(-1,1+1/t)$,
we have $B \ge 0 \Leftrightarrow a \le a_1$, where $a_1 \in (-1,1+1/t)$ 
is the unique value with $g(a_1)=1+1/t$, namely $a_1=1/2t$.
Next we note that $A \ge 0 \Leftrightarrow 1-1/t^2 \ge g(a)^{1-1/\alpha} \Leftrightarrow a \le a_0$,
where $a_0 \in (-1,1+1/t)$ is the unique value with 
$g(a_0)^{1-1/\alpha}=1-1/t^2$.
Since $g(a)^{1-1/\alpha} = 1 + \tfrac{2t+1}{t+1} a + O(a^2)$
we have $a_0 = -1/2t^2 + O(1/t^3)$.
Thus $a_0$ and $a_1$ have the required properties.
It follows that $f(a)$ is maximised at some $b \in [a_0,a_1]$.
By the Mean Value Theorem we have $f(b) = f(0) + bf'(a)$
for some $a$ between $0$ and $b$. 
Since $|a| \le |b| = O(1/t)$, each of the two terms of $A$ are $t^{2-1/\alpha} + O(t^{1-1/\alpha})$
so that $A = O(t^{1-1/\alpha})$.  Similarly, $B = O(1)$,
so $bf'(a) = O(t^{-1-3/\alpha}) = O(t^{-2}) f(0)$.
Therefore \[\lambda_{\alpha}(B_{2t+1}) = 3 t(t+1) f(b)
= (1+O(t^{-2})) 6e(B_{2t+1}) (2t+1)^{-3/\alpha}.\]

It remains to consider the general case of (iii).
Write $n=qr+s$ with $0 \le s < r$.
Then $T_{r,n}$ has $s$ parts of size $q+1$ and $r-s$ parts of size $q$.
By Corollary \ref{cor:sym} we can choose $\vec{w}$ with
$\lambda_{\alpha}(T_{r,n}) = \tau_H(\vec{w},\vec{w})$
such that $w_i=x^{1/\alpha}$ for $i$ in a part of size $q+1$
and $w_i=y^{1/\alpha}$ for $i$ in a part of size $q$,
where $x,y \ge 0$ and $\|w\|_\alpha^\alpha=(q+1)sx+q(r-s)y=1$. Then
\[ \tau_{T_{r,n}}(\vec{w},\vec{w}) 
= (q+1)^2 s(s-1) x^{2/\alpha} + q^2 (r-s)(r-s-1) y^{2/\alpha}
  + 2q(q+1) s(r-s) (xy)^{1/\alpha} . \]
Write \[x = \frac{1+a}{qr+s}, \ y = \frac{1-a(q+1)s/q(r-s)}{qr+s}
\text{ and } f(a) = \tau_{T_{r,n}}(\vec{w},\vec{w})\]
for some $a \in (-1,q(r-s)/(q+1)s)$. Note that $f(0) = 2e(T_{r,n}) n^{-2/\alpha}$.
Next we calculate
\[ f'(a) = \frac{2(q+1)s}{(qr+s)\alpha} 
   \Big( Cx^{1/\alpha-1} - Dy^{1/\alpha-1} \Big), \]
where
\[ C = (q+1)(s-1) x^{1/\alpha} + q(r-s) y^{1/\alpha} \text{ and } 
   D = (q+1)s x^{1/\alpha} + q(r-s-1) y^{1/\alpha}. \]
Note that $D = C + E$, where $E = (q+1) x^{1/\alpha} - q y^{1/\alpha}$.
If $a \geq 0$, then $x \ge y$ which implies that $f'(a) \ge 0$.
On the other hand, if $a$ is such that $E \leq 0$, then $x \leq y$ so that $f'(a) \leq 0$.
We have that $E \le 0 \Leftrightarrow a \le a_0$,
where $a_0$ is the unique value such that
$(1+1/q)^{\alpha} = y/x = \tfrac{1-a_0(q+1)s/q(r-s)}{1+a_0}$, i.e.
\[ a_0 = - \frac{(r-s)q((1+1/q)^{\alpha} -1)}{s(q+1) + (r-s)q (1+1/q)^{\alpha}} .\] 
It follows that $f(a)$ is maximised at some $b \in [a_0,0]$.
By the Mean Value Theorem we have $f(b) = f(0) + bf'(a)$
for some $a$ between $0$ and $b$. 
Since $|a| \le |b| = O(1/q)$,
both $Cx^{1/\alpha-1}$ and $Dy^{1/\alpha-1}$
are $(1+O(1/q)) q(r-1) n^{1-2/\alpha}$,
so $bf'(a) = O(n^{-2/\alpha}) = O(n^{-2}) f(0)$. Therefore 
\[\lambda_{\alpha}(T_{r,n}) = f(b) = (1+O(n^{-2})) 2e(T_{r,n}) n^{-2/\alpha}.\]
\end{proof} 

\begin{proof}[Proof of Corollary \ref{cor:colourcritical}] 
Let $\mathcal{G}_n$ be the set of complete $r$-partite graphs on $n$ vertices. Erd\H{o}s and Simonovits \cite{ES}
showed that there are $\epsilon>0$ and $N>1$
such that for all $n>N$, every $F$-free graph with minimum degree at least $(1-1/r-\epsilon)n$ is $r$-partite. Consequently,
 $\mathcal{G}_n$ is $(F,n,1,1-\epsilon)$-universal for all $n \geq N$.
We will apply Theorem \ref{thm:density} with $\mathcal{F}=\{F\}$.
Let $\delta$ be given as in that theorem. 
We can assume that $N$ is sufficiently large.
We have $\text{ex}(n,F) = e(T_{r,n}) = \tfrac{r-1}{2r}n^2 + O(1)$
and $\pi(F) = \tfrac{r-1}{r}$, so 
\[ | \text{ex}(n,F) - \text{ex}(n-1,F) - \pi(F)n | = O(1) < \delta n,\]
if $n \ge N$ and $N$ is sufficiently large.
Also, by Lemma \ref{lem:radii}(iii) we have
$\lambda_{\alpha}(T_{r,n}) = (1+O(n^{-2})) 2e(T_{r,n}) n^{-2/\alpha}$, so
\[|\lambda_{\alpha}(T_{r,n}) - 2\text{ex}(n,F) n^{-2/\alpha}| 
\leq \delta n^{2(\alpha-1)/\alpha - 1},\]
if $n \ge N$ and $N$ is sufficiently large.  Finally, we must argue that if $G$ is an $r$-partite graph on $n$ vertices, then $\lambda_{\alpha}(G) \le \lambda_{\alpha}(T_{r,n})$. Surprisingly, this fact is not entirely trivial, however, it has recently been proved (even more generally for $r$-partite $k$-uniform hypergraphs) in \cite{KNY}.
Thus Theorem \ref{thm:density} implies the Corollary.
\end{proof} 

\begin{proof}[Proof of Corollary \ref{cor:fano}] 
Let $\mathcal{G}_n$ be the set of 2-colorable 3-uniform hypergraphs on $n$ vertices. F\"uredi and Simonovits \cite{FS}
showed that there are $\epsilon>0$ and $N>1$
such that $\mathcal{G}_n$ is $(Fano,n,1,1-\epsilon)$-universal for all $n \geq N$.
We will apply Theorem \ref{thm:density} with $\mathcal{F}=\{Fano\}$.
We can assume that $N$ is sufficiently large.
We have $\pi(Fano) = \tfrac{3}{4}$ and
$\text{ex}(n,Fano)-\text{ex}(n-1,Fano) = e(B_n)-e(B_{n-1})
= \tfrac{3}{8}n^2 + O(n)$, so 
\[ \left| \text{ex}(n,Fano) - \text{ex}(n-1,Fano) - \pi(Fano) \binom{n}{2} \right| 
  = O(n) < \delta n^2,\]
if $n \ge N$ and $N$ is sufficiently large.
Also, by Lemma \ref{lem:radii}(ii) we have
$\lambda_{\alpha}(B_n) = (1+O(n^{-2})) 6e(B_n) n^{-3/\alpha}$, so
\[|\lambda_{\alpha}(B_n) - 6\text{ex}(n,Fano) n^{-3/\alpha}| 
\leq \delta n^{3(\alpha-1)/\alpha - 1},\]
if $n \ge N$ and $N$ is sufficiently large.
It suffices to show that if $\chi(G)=2$, then 
$\lambda_{\alpha}(G) \le \lambda_{\alpha}(B_n)$, and this follows from a result of \cite{MT} for $\alpha=2$ and from \cite{KNY} for all $\alpha>1$. 
Thus Theorem \ref{thm:density} implies the Corollary.
\end{proof} 

\begin{proof}[Proof of Corollary \ref{cor:kk}] 
Suppose that $e(\partial H) < \tbinom{x}{k-1}$.
By Lov\'asz \cite[Ex 13.31(b)]{Lo} we have $e(H) < \tbinom{x}{k}$.
But then Lemma \ref{lem:edgebound} gives
$\lambda_{\alpha}(H) < \Big( k!\tbinom{x}{k} \Big)^{1-1/\alpha}$, contradiction.
\end{proof} 

\begin{proof}[Proof of Corollary \ref{cor:intersecting}] 
We will apply Theorem~\ref{thm:strongstab} to $\mathcal{F} = \{F_0,\dots,F_{t-1}\}$, where $F_i$ is
the $k$-uniform hypergraph with $2$ edges that intersect in $i$ vertices.  Let $\mathcal{G}_n =
\{S^k_{t,n}\}$.  Erd\H{o}s, Ko and Rado \cite{EKR} showed that
$\text{ex}(n,\mathcal{F})=e(S^k_{t,n})$ for large $n$, and moreover there is some absolute constant
$C$ such that $\mathcal{G}_n$ is an $(\mathcal{F},n,0,C/n)$-universal family.  By Lemma
\ref{lem:radii}(i) we have $\lambda_{\alpha}(\mathcal{G}_n) = \lambda_{\alpha}(S^k_{t,n}) =
\Theta(n^{(k-t)(\alpha-1)/\alpha})$.  We now apply Theorem~\ref{thm:strongstab}.  First, we check
\begin{align*}
  \frac{C}{n} < \frac{\lambda_{\alpha}(\mathcal{G}_n)^{\alpha/(\alpha-1)}}{k!ex(n,\mathcal{F})} =
  \frac{\Theta(n^{k-t})}{k! \binom{n-t}{k-t}} = \Theta(1),
\end{align*}
which is true for large $n$.  Therefore, for large $n$, we can apply Theorem~\ref{thm:strongstab} to
obtain that for any $k$-uniform, $\mathcal{F}$-free hypergraph $H$ on $n$ vertices, we have that
$\lambda_{\alpha}(H) \leq \lambda_{\alpha}(S^k_{t,n})$.  Furthermore, there is a constant $c =
c(k,t)$ such that if $\lambda_{\alpha}(H) > cn^{-(\alpha-1)/\alpha} \lambda_{\alpha}(S^k_{t,n})$
then $\lambda_{\alpha}(H) > (Cn^{-1}k!\text{ex}(n,\mathcal{F}))^{(\alpha-1)/\alpha}$, so that
Theorem~\ref{thm:strongstab} implies that if $\lambda_{\alpha}(H) >
cn^{-(\alpha-1)/\alpha)}\lambda_{\alpha}(S^k_{t,n})$, then $H \subseteq S^k_{t,n}$, i.e.\ $H$ is a
star.
\end{proof} 

\section{Concluding remarks} 
\label{sec:concluding}

%
%

In this paper we have given two general criteria that can be applied to obtain
$\alpha$-spectral versions of a variety of Tur\'an type problems.
We have illustrated some such applications, but we have not attempted to be exhaustive:
there are several other examples which can no doubt be treated by the same method,
although it would be laborious to give the details.
On the other hand, it would be more interesting to obtain $\alpha$-spectral results
for hypergraphs where the spectral extremal example differs from the usual extremal example.
For example, Nikiforov \cite{N2,N4} showed that when $n$ is odd, the maximal spectral radius
of a $C_4$-free graph on $n$ vertices is uniquely achieved by the `friendship graph', which consists
of $(n-1)/2$ triangles intersecting in a single common vertex. This is very different from
the extremal example for the maximum number of edges in a $C_4$-free graph on $n$ vertices:
F\"uredi \cite{Fu} showed that for large $n$ of the form $q^2+q+1$ this is uniquely achieved
by the polarity graph of a projective plane.
 
In all our results we assumed that $\alpha>1$. 
This assumption is necessary, as different behaviour emerges at $\alpha=1$.
In this case, the $\alpha$-spectral radius is the well-studied hypergraph lagrangian.
Consider for example the $3$-uniform hypergraph $F_5 = \{123,124,345\}$.
The Tur\'an problem was solved by Frankl and F\"uredi \cite{FF83}:
they showed that for large $n$ the unique largest $F_5$-free $3$-uniform hypergraph
on $n$ vertices is the balanced complete tripartite $3$-uniform hypergraph $T^3_n$.
A stability result was obtained by Keevash and Mubayi \cite{KM}
(who also improved the bound on $n$) and the optimum bound on $n$
was obtained by Goldwasser \cite{G}. Similar arguments along the lines in this paper 
would no doubt enable one to deduce the $\alpha$-spectral version from Theorem \ref{thm:density}
when $\alpha>1$. However, the corresponding result does not hold when $\alpha=1$.
Indeed, in this case $\lambda_{1}(T^3_n) = 2/9 + o(1)$ for large $n$,
but the $F_5$-free $3$-uniform hypergraph $H$ equal to $K^3_4$ with $n-4$ isolated vertices 
has $\lambda_{1}(H) = 3/8$.

The line of research started in this paper can be viewed as a generalisation of that proposed in
\cite{HK}, namely to determine the maximum lagrangian for any specified property of hypergraphs.
Such questions go back to Frankl and F\"uredi \cite{FF89}, who considered the question of maximising
the lagrangian of an $k$-uniform hypergraph with a specified number of edges. They conjectured that
initial segments of the colexicographic order are extremal. Many cases of this have been proved by
Talbot \cite{T}, but the full conjecture remains open. Here we propose the natural generalisation of
this conjecture, namely that among $k$-uniform hypergraphs with a specified number of edges, initial
segments of the colexicographic order maximise the $\alpha$-spectral radius, for any $\alpha \ge 1$.
Lemma \ref{lem:edgebound} gives an asymptotic form of this conjecture, as for any $k$-uniform 
hypergraph $H$ with $e(H) = (1+o(1))\tbinom{n}{k}$ it implies
\[\lambda_{\alpha}(H) \le (1+o(1)) \Big( k!\tbinom{n}{k} \Big)^{1-1/\alpha} 
= (1+o(1))\lambda_{\alpha}(K^k_n).\]

A final question is whether Corollary \ref{cor:intersecting} can be extended to cover all values of
$n$, possibly with other extremal examples, as in the Complete Intersection Theorem of Ahlswede and
Khachatrian.  We expect that the $\alpha$-spectral result may differ from the extremal
result. Indeed, consider the case $k=2$, $t=1$, $n=4$ and $\alpha=2$. For the extremal problem, the
maximum of $3$ edges is achieved both by the star $K_{1,3}$, and the triangle $K_3$ (plus an
isolated vertex). However, $\lambda_{2}(K_{1,3}) = \sqrt{3} < 2 = \lambda_{2}(K_3)$.

\section{Acknowledgments}
We are very grateful to the referees for their careful reading of the manuscript and suggesting a shorter proof of Lemma 5.



\begin{thebibliography}{99}

\bibitem{AK} R. Ahlswede and L.H. Khachatrian,  
The complete intersection theorem for systems of finite sets,
{\em European J. Combin.} {\bf 18} (1997), 125--136.

\bibitem{AES} B. Andr\'asfai, P. Erd\H{o}s and V. T. S\'os,
On the connection between chromatic number, maximal clique and minimal degree of a graph.
{\em Disc. Math.} {\bf 8} (1974), 205--218.

\bibitem{Bollobas} B. Bollob\'{a}s, Modern graph theory, {\em Graduate Texts in Mathematics} {\bf
184} (1998), Springer-Verlag, New York.

\bibitem{EKR} P. Erd\H{o}s, C. Ko and R. Rado,
Intersection theorems for systems of finite sets, 
{\em Quart. J. Math. Oxford Ser.} {\bf 12} (1961), 313--320.

\bibitem{ES} P. Erd\H{o}s and M. Simonovits,
On a valence problem in extremal graph theory, 
{\em Disc. Math.} {\bf 5} (1973), 323--334.

\bibitem{FF83} P. Frankl and Z. F\"uredi,
A new generalization of the Erd\H{o}s-Ko-Rado theorem,
{\em Combinatorica} {\bf 3} (1983), 341--349.

\bibitem{FF89} P. Frankl and Z. F\"uredi, Extremal problems whose solutions are the blow-ups of the
small Witt-designs, \emph{J. Combin. Theory Ser. A} {\bf 52} (1989), 129--147.

\bibitem{Fr} J. Friedman, Some graphs with small second eigenvalue,
{\em Combinatorica} {\bf 15} (1995), 31--42.

\bibitem{FW} J. Friedman and A. Wigderson, On the second eigenvalue of hypergraphs,
{\em Combinatorica} {\bf 15} (1995), 43--65.

\bibitem{Fu} Z. F\"uredi,
Quadrilateral-free graphs with maximum number of edges, unpublished (1998), 
available at http://www.math.uiuc.edu/$\sim$ z-furedi/publ.html .

\bibitem{FS} Z. F\"uredi and M. Simonovits,
Triple systems not containing a Fano configuration,
{\em Combin. Probab. Comput.} {\bf 14} (2005), 467--484.

\bibitem{G} J. Goldwasser, On the Tur\'an Number of $\{123,124,345\}$, manuscript.

\bibitem{HK} D. Hefetz and P. Keevash, 
A hypergraph Tur\'an theorem via lagrangians of intersecting families, submitted.


\bibitem{KNY} L. Kang, V. Nikiforov, X Yuan, The $p$-spectral radius of $k$-partite and $k$-chromatic uniform hypergraphs, preprint: ArXiv 1402.0442


\bibitem{K} P. Keevash, Hypergraph Tur\'an problems, {\em Surveys in Combinatorics},
Cambridge University Press, 2011, 83--140.

\bibitem{KM}
P. Keevash and D. Mubayi,
Stability theorems for cancellative hypergraphs,
{\em J. Combin. Theory Ser. B} {\bf 92} (2004), 163--175.

\bibitem{KS}
P. Keevash and B. Sudakov, The Tur\'an number of the Fano plane,
{\em Combinatorica} {\bf 25} (2005), 561--574.

\bibitem{MT} D. Mubayi, J. Talbot, Extremal problems for $t$-partite and $t$-colorable hypergraphs, Electronic Journal of Combinatorics 15 (2008), no. 1, Research Paper 26, 9 pp

\bibitem{Lo} L. Lov\'asz,
{\em Combinatorial Problems and Exercises},
North-Holland, Amsterdam, 1993.

\bibitem{N} V. Nikiforov, Some new results in extremal graph theory, {\em Surveys in Combinatorics}
Cambridge University Press, 2011, 141--181.

\bibitem{N2} V. Nikiforov, Bounds on graph eigenvalues II, 
{\em Linear Algebra Appl.} {\bf 427} (2007), 183--189.


\bibitem{N4} V. Nikiforov,
The maximum spectral radius of C4-free graphs of given order and size, 
{\em Linear Algebra Appl.} {\bf 430} (2009), 2898--2905.

\bibitem{N5} V. Nikiforov,
Spectral saturation: inverting the spectral Tur\'{a}n theorem,
{\em Electronic J. Combin.} {\bf 15} (2009), R33.

\bibitem{Sim} M. Simonovits,
A method for solving extremal problems in graph theory, stability problems,
{\em Theory of Graphs (Proc. Colloq.\ Tihany, 1966)}, Academic Press,
New York, and Akad.\ Kiad\'o, Budapest, 1968, 279--319.

\bibitem{Sos} V. S\'os,
Remarks on the connection of graph theory, finite geometry and block designs,
{\em Teorie Combinatorie}, Tomo II, Accad. Naz. Lincei, Rome, 1976, 223--233.

\bibitem{T} J. Talbot, Lagrangians of hypergraphs, {\em Combin. Probab. Comput.}
{\bf 11} (2002), 199--216.

\bibitem{T41} P. Tur\'an,
On an extremal problem in graph theory (in Hungarian),
{\em Mat. Fiz. Lapok} {\bf 48} (1941), 436--452.

\bibitem{T61} P. Tur\'an, Research problem,
{\em K\"ozl MTA Mat. Kutat\'o Int.} {\bf 6} (1961), 417--423.

\bibitem{W} R. M. Wilson,
The exact bound in the Erd\H{o}s-Ko-Rado theorem,
{\em Combinatorica} {\bf 4} (1984), 247--257.

\end{thebibliography}
\end{document}